\documentclass[journal]{IEEEtran}
\usepackage{color} 
\usepackage{ifpdf}

\usepackage{cite}

\ifCLASSINFOpdf
  \usepackage[pdftex]{graphicx}

\else
 
\fi

\usepackage{graphicx}
\usepackage{epstopdf}
\usepackage{subfigure}
\usepackage{setspace}
\usepackage{amsmath}
\usepackage{amsfonts}
\usepackage{amssymb}
\ifCLASSOPTIONcompsoc
  \usepackage[caption=false,font=normalsize,labelfont=sf,textfont=sf]{subfig}
\else
  \usepackage[caption=false,font=footnotesize]{subfig}
\fi

\usepackage{fixltx2e}
\usepackage{footnote}
\makesavenoteenv{tabular}

\ifCLASSOPTIONcaptionsoff
  \usepackage[nomarkers]{endfloat}
 \let\MYoriglatexcaption\caption
 \renewcommand{\caption}[2][\relax]{\MYoriglatexcaption[#2]{#2}}
\fi

\usepackage{url}

\allowdisplaybreaks[4]

\usepackage{comment}
\usepackage{amsthm}
\usepackage{ifthen}

\newtheorem{proposition}{Proposition}

\newboolean{showcomments}
\setboolean{showcomments}{false}
\newcommand{\fliu}[1]{\ifthenelse{\boolean{showcomments}}
	{ \textcolor{red}{(FL:  #1)}}{}}
\newcommand{\gjp}[1]{\ifthenelse{\boolean{showcomments}}
	{ \textcolor{blue}{(Gjp:  #1)}}{}}

\def\ba{\begin{array}}
	\def\ea{\end{array}}
\newcommand{\beq}{\begin{equation}}
\newcommand{\eeq}{\end{equation}}
\newcommand{\bq}{\begin{eqnarray}}
\newcommand{\eq}{\end{eqnarray}}
\newcommand{\bqn}{\begin{eqnarray*}}
	\newcommand{\eqn}{\end{eqnarray*}}
\newcommand{\bee}{\begin{enumerate}}
	\newcommand{\eee}{\end{enumerate}}
\newcommand{\bi}{\begin{itemize}}
	\newcommand{\ei}{\end{itemize}}

\begin{document}
	
	\title{Mitigating Blackout Risk via Maintenance : Inference from Simulation Data }

	\author{Jinpeng~Guo,~\IEEEmembership{Student Member,~IEEE, }
		Feng~Liu,~\IEEEmembership{Member,~IEEE, }
        Xuemin~Zhang,~\IEEEmembership{Member,~IEEE, }
        Yunhe~Hou,~\IEEEmembership{Senior Member,~IEEE, }        
		and~Shengwei~ Mei,~\IEEEmembership{Fellow,~IEEE}
		\thanks{Manuscript received XXX, XXXX; revised XXX, XXX. \textit{(Corresponding author: Feng Liu)}.}
       	\thanks{J. Guo, F. Liu, X. Zhang and S. Mei are with the State Key Laboratory of Power Systems, the Department of Electrical Engineering, Tsinghua University, Beijing, 100084, China  
       		.} 
       	\thanks{Y. Hou is with the Department of Electrical and Electronic Engineering,
       		The University of Hong Kong, Hong Kong}

        }
	
	\markboth{REPLACE THIS LINE WITH YOUR PAPER IDENTIFICATION NUMBER (DOUBLE-CLICK HERE TO EDIT}%
	{Shell \MakeLowercase{\textit{et al.}}:  Bare Demo of IEEEtran.cls for Journals}
	
	\maketitle
	
	\begin{abstract}
Whereas maintenance has been recognized as an important and effective means for risk management in power systems, it turns out to be intractable if cascading blackout risk is considered due to the extremely high computational complexity. In this paper, based on the inference from the blackout simulation data, we propose a methodology to efficiently identify the most influential component(s) for mitigating cascading blackout risk in a large power system. To this end, we first establish an analytic relationship between maintenance strategies and blackout risk estimation by inferring from the data of cascading outage simulations.  Then we formulate the component maintenance decision-making problem as a nonlinear 0-1 programming. Afterwards, we quantify the credibility of blackout risk estimation, leading to an adaptive method to determine the least required number of simulations, which servers as a crucial parameter of the optimization model. Finally, we devise two  heuristic algorithms to find approximate optimal solutions to the model with very high efficiency. Numerical experiments well manifest the  efficacy and high efficiency of our methodology.    	
	\end{abstract}
	
		
	\IEEEpeerreviewmaketitle

	\section{Introduction}
	\IEEEPARstart{U}{nder} certain conditions, system  disturbances or component outages in a power system can initiate a sequence of component failures, i.e., cascading outages, leading to serious consequence, even catastrophic blackouts. Although the probability of such catastrophic blackouts is  tiny,  theoretical research and practical experience have revealed that such blackout risk cannot be ignored\cite{r1,r2,r3,r4}. 
	
	Intuitively, maintaining component  is  seen as an effective means  to mitigate cascading blackout risk since it  directly reduces the  probability of component failures \cite{r5,r6}. 
	 In practice, whereas there are usually a huge number of components in a power system, a few of them have much larger influence on the blackout risk than the others\cite{r7,r8,r9}. Therefore, choosing those most influential components for maintenance can greatly facilitate effective mitigation of  blackout risk with limited resource. This idea is not new, which has been extensively deployed in conventional reliability-centered maintenance (RCM) or risk-based maintenance (RBM). 	In \cite{r10}, a risk-based resource optimization model for transmission system maintenance is proposed, where both maintenance strategies and corresponding risk reductions serve as input data. From a different perspective, risk could be defined and calculated via scenario enumeration. And those components associated with high-risk scenarios are selected for maintenance in \cite{r11}. More rigorously,  \cite{r12} employs a 0-1 integer programming to optimize  system risk with  resource limitation in power systems, where the system risk is defined as the sum of component risks and is calculated by enumeration. A similar optimization model is presented in \cite{r6}. In all these methods, maintenance strategies are selected based on risk evaluation with respect to the considered maintenance strategies.
 	
 	The aforementioned methods have been demonstrated to work well in both RCM and RBM.  Unfortunately, when the cascading blackout is considered, it is not the case. The main reason relies on the matter of fact that it is extremely computational complex to evaluate cascading blackout risk with respect to  maintenance strategies. 
	As well known, the propagation of a cascading outage is a complicated dynamic process involving  various random factors, making it impossible to evaluate blackout risk analytically.  Hence Monte Carlo (MC) method is often employed to estimate cascading blackout risk as a surrogate. To achieve a credible estimation, however, MC method needs to generate a large number of samples via cascading blackout simulations \cite{r13,r14,r15}, which is extremely time consuming\cite{r16}. 
	If a large number of components are taken into account, MC method may fail due to the notorious  ``curse of dimensionality''. When maintenance is considered, the problem turns out to be much worse. Since there is lack of analytic relationship to bridge the estimated blackout risks with component failure probabilities that vary with maintenance strategies,  samples used to estimate the blackout risk with respect to a specific maintenance strategy cannot be used for another. That implies, whenever the maintenance strategy changes, all blackout samples have to be completely regenerated by conducting blackout simulations. Considering the huge number of components and possible maintenance strategies, both sample generation and blackout risk estimation are extremely time-consuming, making the corresponding maintenance optimization problem intractable.    
	
	Regarding to this obstacle,  \cite{r17} exploits the underlying information of cascading blackouts hidden in the data of cascading blackout simulations,  leading to a semi-analytic method to characterize the relationship between (unbiased) blackout risk estimation and component failure probabilities. It opens up the possibility to directly estimate blackout risk under varying component failure probabilities, with no need of regenerating any samples. That essentially motivates an efficient risk evaluation approach to cope with varying maintenance strategies. In this paper, we extend the work to address the optimal component maintenance problem considering cascading blackout risk. Major contributions of our work are threefold:
		\begin{enumerate}
		\item  Inferring from cascading blackout simulation data, an analytic relationship between estimated blackout risk and maintenance strategy is established. Based on it, we formulate the optimal component maintenance problem as a  nonlinear 0-1 programming. Particularly, the evaluation of cascading blackout risk with respect to varying maintenance strategies, at the first time, is explicitly  expressed based on simulation data in the model. 
		\item  In order to guarantee the credibility of the blackout risk evaluation when solving the proposed optimization model, we propose an adaptive method to determine an appropriate sample size via credibility analyses based on the inference from simulation data. 
		\item  The proposed optimization model is a high-dimensional nonlinear 0-1 programming, which is NP-hard  with a  complicated objective function. To solve this nontrivial optimization problem, we propose two simple heuristic algorithms to search nearly-optimal solutions to the problem with very high efficiency.
	\end{enumerate}	  
	
	The rest of this paper is organized as follows. Section II gives  basic definitions and notations, based on which the optimization problem is formulated. In Section III, according to the credibility analyses of  blackout risk estimation, the critical parameter, sample size, is determined. Two high-efficiency heuristic algorithms and the holistic procedure are presented in Section IV. In Section V, case studies are presented. Finally, Section VI concludes the paper with remarks.

\section{Problem Statement and Formulations}

In this work, we aim at  mitigating  cascading blackout risk via maintaining a few components in the system. Basically, the influence of maintenance on component failure probabilities can be estimated based on the historical data and experience. Therefore, the critical issue is how to characterize the relationship between component failure probabilities and  blackout risk (or more precisely, the estimation of blackout risk). As long as the relationship is obtained, it is straightforward to decide maintenance strategies  by enumeration or optimization. Obviously, an analytic relationship is preferable for  maintenance strategy optimization. In this context, we employ the method given in \cite{r17} to further analytically characterize the relationship between estimated blackout risk and maintenance strategies through the inference from  blackout simulation data, and formulate the optimization problem explicitly.     
          
\subsection{Definitions}
We start with the basic definitions of  cascading outages, component failure probability function as well as blackout risk.

In light of \cite{r16}, an $n$-stage cascading outage can be defined as a Markov sequence denoted by 
\begin{equation}\nonumber
Z:= \{ {X_0, X_1, ..., X_j, ..., X_n}, X_j\in \mathcal{X}, \forall j\in \mathbb{N}  \}
\end{equation}
with respect to a probability series, $g(Z)$. Here, $\mathbb{N}:=\{1,2,\cdots, n\}$ is the set of cascading stages; $j$ is the stage label; $X_j$ is the state variables of the system at stage $j$, which can be ON/OFF statues of components, power injections at each bus, etc. Particularly, $X_0$ is the initial system state of the cascading outage, which is assumed to be deterministic. The state space $\mathcal{X}$ is assumed to be finite. Regarding a specific cascading outage, $z=\{x_0, ..., x_j, ..., x_n\}$. And its probability (series) is denoted by $g(z):=g(x_0, ..., x_j, ..., x_n)$. At  stage $j$, the failure probability of component $k$ is denoted by
\begin{equation} \label{eq2}
	\varphi_k(x_j):=\mathbf{Pr}(\text{component}\; k\; \text{fails at}\; x_j \; ) 
\end{equation}
where $\varphi_k$ is referred to as \emph{component failure probability function}. It is determined by inherent characteristics of the component, e.g., component type, operating condition, etc. If a certain component $k$ is maintained,  the failure probability function will change accordingly. We denote the failure probability function of component $k$ after maintenance by $\bar{\varphi}_k$.

Noting that load shedding can be involved in a cascading outage, we use a general definition of cascading blackout risk associated with load shedding \cite{r16}. Specifically, denote $Y$ as the load shedding of a cascading outage. It is a random variable and can be further regarded as a function of the corresponding cascading outage, i.e., $Y=h(Z)$. Then the blackout risk with respect to $g(Z)$ and a given load shedding level $Y_0$ is 
\begin{equation} \label{eq4}
		R_g(Y_0)  :=  \mathbb{E}(Y\cdot\delta_{\{Y\ge{Y_0}\}})=\sum\limits_{z \in \mathcal{Z}} {g(z)h(z){\delta _{\{ h(z) \ge {Y_0}\} }}} 
		\end{equation}		
where $\mathcal{Z}$ is the set of all possible cascading outages; $\delta_{\{Y\ge{Y_0}\}}$  an indicator function of the events $\{Y\ge{Y_0}\}$, given by
\begin{equation*} 
\delta_{\{Y\ge{Y_0}\}}:=\left\{\begin{array}{lll}
1 & &\text{if}\quad Y\ge{Y_0};\\
0 & &\text{otherwise}.
\end{array} \right.
\end{equation*}

The cascading blackout risk defined in Eq. \eqref{eq4} is indeed the expectation of load shedding beyond the given level, $Y_0$. When $Y_0=0$, it is  the traditional definition of blackout risk. When $Y_0>0$, it is the risk of those events with serious consequences, specifically, with  a load shedding greater than $Y_0$. 
\subsection{Analytic Relationship Inference from Simulation Data}
Invoking the Markov property and conditional probability formula, $g(z)$ can be rewritten as
\begin{equation} \label{eq1}
g(z) = g({x_n}, \cdots, {x_1},{x_0})={\prod\limits_{j=0}^{n-1}{g_{j+1}(x_{j+1}|x_{j})}}
\end{equation} 
where $g_{j+1}({x_{j+1}}|{x_{j}})$ represents the transition probability from state $x_j$ to state $x_{j+1}$. Considering the failure components at stage $j$, \eqref{eq1} is equivalent to
		\begin{equation} \label{eq3}
		g(z)=\prod\limits_{j=0}^{n-1}\left[\prod\limits_{k \in {F_j}} {{\varphi_{k}}({x_j})} \cdot \prod\limits_{k \in {\bar{F}_j}} {(1 - {\varphi _{k}}({x_j}))} \right]		
		\end{equation}	
where $F_j$ is the component set consisting of the components that are defective at $x_{j+1}$ but work normally at $x_{j}$, while $\bar{F}_j$ consists of components that work normally at $x_{j+1}$. 

Since maintenance only influences some components in the system, we consider the items related to a specific component $k \in K$ in \eqref{eq3} and define
\begin{small}	
	\begin{equation} \label{eq8}
	\Gamma(\varphi_{k},z) := \left\{ {\begin{array}{*{20}{c}}
		\prod\limits_{j=0}^{n-1}{(1 - {\varphi _k}({x_j}))}&: \text{if}\quad n_k=n\\
		{\varphi _k}({x_{n_k}}) \prod\limits_{j=0}^{n_k-1}{(1 - {\varphi _k}({x_j}))}&: \text{otherwise}\\
		\end{array}} \right.
	\end{equation}
\end{small}
where $K$ is the set including all components; $n_k$ is the stage at which component $k$ fails. Particularly, if component $k$ does not fail in the whole cascading process, let $n_k:=n$. Then

\begin{equation} \label{eq9}
g(z)=\prod\limits_{k\in K}{\Gamma(\varphi_k,z)}
\end{equation}
Substituting \eqref{eq9} into \eqref{eq4} yields 

\begin{equation} \label{eq5}
R_g(Y_0)=\sum\limits_{z \in \mathcal{Z}} {\left(h(z){\delta _{\{ h(z) \ge {Y_0}\} }}\cdot\prod\limits_{k\in K}{\Gamma(\varphi_k,z)}\right)}
\end{equation}

Eq.\eqref{eq5} reveals the inherent relationship between $R_g(Y_0)$ and $\varphi_k$. However, since $|\mathcal{Z}|$ and $|K|$ \footnote{$|\cdot|$ is the operator stands for the cardinality of a set.} may be very large, it cannot be directly used in optimization. Alternatively, we use an approximation of $R_g(Y_0)$ in terms of a set of samples, which is given by
\begin{equation} \label{eq6.1}
\hat{R}_g(Y_0)=\frac{1}{N}\sum\limits_{i = 1}^N {{h(z^i)}{\delta _{\{ {h(z^i)} \ge {Y_0}\} }}}
\end{equation}
In \eqref{eq6.1}, $N$ is the number of independent identically distributed (i.i.d.) samples which are generated using specific blackout simulation models with respect to $g(Z)$ (or more precisely, $\varphi_k, k \in K$); $z^i= \{x^i_0, ..., x^i_j, ..., x^i_{n^i}\}$ is the $i$-th sample, where $n^i$ is the number of stages in $i$-th sample. Particularly, we define $Z_g:=\{z^i,i=1,\cdots,N\}$,which stands for the sample set generated with respect to $g(Z)$.
\fliu{so what does $Z_g$ mean? Does it stand for the sample set generated with respect to $g(Z)$?}

When some  $\varphi_k$ change due to maintenance, $g(Z)$ will be converted into another probability series. We  denote the new probability series after maintenance by $f(Z)$. Regarding $Z_g$,  \cite{r17} says that the unbiased estimation of $R_f(Y_0)$ is given by 
\begin{equation} \label{eq6}
\hat{R}_f(Y_0)=\frac{1}{N}\sum\limits_{i = 1}^N {w(z^i){h(z^i)}{\delta _{\{ {h(z^i)} \ge {Y_0}\} }}}
\end{equation} 
where $w(z^i)$ is the sample weight of $z^i$ and is defined as 
\begin{equation} \label{eq7}
w(z^i):=\frac{f(z^i)}{g(z^i)}, \quad (\forall z^i\in\mathcal{Z})
\end{equation} 

According to \eqref{eq6}, it is interesting that  $\hat{R}_f(Y_0)$ only requires the information of  $z^i$ which is generated in the blackout simulations with respect to $g(Z)$.  It implies that, when $g(Z)$ changes to $f(Z)$, the blackout risk can be directly obtained using the information of $
Z_g$, other than regenerating samples via blackout simulations. Therefore, it enables a way to explicitly express the blackout risk with respect to maintenance strategies,  which will greatly facilitate the formulation of optimal maintenance decision-making problem, as we explain.

\subsection{Formulation of Maintenance Strategy Optimization}
We use a binary variable $m_k$ to represent the maintenance status of component $k$. If component $k$ is maintained, $m_k=1$; otherwise, $m_k=0$. Then we use a
vector $M:=\{m_k,k \in K\}$ to represent the maintenance strategy. For a specific sample $z^i$, we have  
\begin{equation} \label{eq10}
f(z^i)=\prod\limits_{k\in K}{\left[m_k\Gamma(\bar{\varphi}_k,z^i)+(1-m_k)\Gamma(\varphi_k,z^i)\right]} 
\end{equation}

Substituting \eqref{eq10} into \eqref{eq6} yields  \fliu{Please carefully check the following question. It looks problematic.}
\begin{equation} \label{eq11}
\begin{array}{ll}
\hat{R}_f(Y_0)&\\
=\frac{1}{N}\sum\limits_{i = 1}^N {{\frac{ \prod\limits_{k\in K}{\left[m_k\Gamma(\bar{\varphi}_k,z^i)+(1-m_k)\Gamma(\varphi_k,z^i)\right]}  }{ \prod\limits_{k\in K}{\Gamma(\varphi_k,z^i)}  } h(z^i)}{\delta _{\{ {h(z^i)} \ge {Y_0}\} }}}\\
=\frac{1}{N}\sum\limits_{i = 1}^N {\left[{ \left(\prod\limits_{k\in K^*}{ 1+m_k\left(\frac{\Gamma(\bar{\varphi}_k,z^i)}{\Gamma(\varphi_k,z^i)}-1\right) }\right)   h(z^i)}{\delta _{\{ {h(z^i)} \ge {Y_0}\} }}\right]}\\
\end{array}
\end{equation} 
where $K^*$ is the set including all the components which are available for maintenance.  Obviously, $K^* \subseteq  K$ and $m_k=0, \forall k \notin K^*$.
   
Eq.\eqref{eq11} reveals the relationship between estimated blackout risk and maintenance strategies. Moreover, it essentially  provides the estimation of cascading blackout risk with an explicit expression.
To minimize the  blackout risk with a limited number of components considered in maintenance, an optimization problem is formulated as
\fliu{please check the equation below}
    \begin{equation} \label{eq12}
		\begin{array}{l}
        \min \limits_{m_k }  \frac{1}{N}\sum\limits_{i = 1}^N {\left[{ \left(\prod\limits_{k\in K^*}{ 1+m_k\left(\frac{\Gamma(\bar{\varphi}_k,z^i)}{\Gamma(\varphi_k,z^i)}-1\right) }\right)   h(z^i)}{\delta _{\{ {h(z^i)} \ge {Y_0}\} }}\right]} \\
        \\
        s.t.  \sum\limits_{k\in K^*}{m_k} \leq M_{max}  \\

		\end{array}
	\end{equation}
where $m_k, k \in K^*$ are the decision variables. $M_{max}$ is a predefined parameter that stands for the maximum number of components considered in maintenance. $\Gamma(\varphi,z^i)$, $\Gamma(\bar{\varphi}_k,z^i)$, $h(z^i)$ and $\delta _{\{ {h(z^i)} \ge {Y_0}\} }$ are  variables rely on samples. For simplicity, we define an $N$-dimension vector $C$ and two $|K^*| \times N$ matrices, $P$ and $Q$, which are defined by
\begin{equation*}
\begin{array}{rlll}
 C_i&:=&h(z^i)\delta _{\{ {h(z^i)} \ge {Y_0}\} }& \forall i  \\ P_{ki}&:=&\Gamma({\varphi}_k,z^i)&\forall i,k\\
 Q_{ki}&:=&\Gamma(\bar{\varphi}_k,z^i) &\forall i,k
\end{array}
\end{equation*}

The optimization model \eqref{eq12} is a high-dimensional 0-1 programming. To solve this nontrivial optimization problem, the following two issues should be well addressed.

\subsubsection{Determining appropriate sample size, $N$}
In the optimization model \eqref{eq12}, the sample size, i.e., $N$, is a crucial parameter, since the  objective function  is an unbiased estimation of  blackout risk and the estimation error relies on $N$. To guarantee the credibility of the estimation, a large enough $N$ is necessary. On the other hand, a too large $N$ will remarkably increase the computational burden. Unfortunately, 
it is not trivial to determine an appropriate sample size. In this regard, we analyze the variance of blackout risk estimation based on the inference from simulation data, leading to an adaptive method to find a suitable sample size  that achieves a good trade-off between estimation accuracy and computation complexity. Details will be given in Sections \ref{sec:credibility} and \ref{subsec: procedure}.
\subsubsection{Reducing computational complexity}
Even if an appropriate sample size $N$ has been determined, solving the optimization problem \eqref{eq12} still faces with great challenge. Note that the number of decision variables in \eqref{eq12}, which equals $|K^*|$,  could be very large. Moreover, there are $2^{|K^*|}-1$ product terms of decision variables in the objective function. Therefore, optimization problem \eqref{eq12} is a high-dimensional  0-1 integer programming. It is NP-hard, where enumeration and conventional optimization algorithms (e.g., branch and bound) are proved powerless.  In this context, we devise two simple but efficient algorithms to find nearly-optimal solutions of \eqref{eq12} in a heuristic manner. Details  will be given in Section \ref{sec:alg}. 

\section{Determining $N$ Based on Credibility  Analysis}\label{sec:credibility}
Due to the inherent uncertainty of samples in $Z_g$, the estimated blackout risk $\hat{R}_f(Y_0)$ given in \eqref{eq6} always has certain estimation error. Generally, increasing sample size helps reduce estimation error. However, for determining an appropriate sample size of $Z_g$, i.e., $N$, we need to characterize the influence of $N$ on the estimation error via variance analysis based on the inference from simulation data. 

\subsection{Variance-based Credibility Analyses}
We first consider the relative error bound of \eqref{eq6} with respect to a specific $Z_g$. Specifically, denote $\epsilon$ as the relative error bound of $\hat{R} _f(Y_0)$ with a confidence level $\beta$. Then $\epsilon$ should make the probability of $R_f(Y_0)$ within $I:=[(1-\epsilon)\hat{R}_{f}(Y_0),(1+\epsilon)\hat{R}_{f}(Y_0)]$ with a probability $\beta$, i.e.,
 \begin{equation}\label{eq17}
\mathbf{Pr}\left( \frac{|R_{f}(Y_0)-\hat{R}_{f}(Y_0)|}{R_{f}(Y_0)} \le \epsilon \right) = \beta 
\end{equation}

Here $2\epsilon$ is defined as the relative length of error bound, which is used to quantify the credibility of  blackout risk estimation, as we explain. Note that \eqref{eq17}  is equivalent to
\begin{equation}\label{eq18}
\mathbf{Pr}\left( -\frac{\epsilon R_{f}(Y_0)}{\sqrt{D_f(Y_0)}} \le \frac{R_{f}(Y_0)-\hat{R}_{f}(Y_0)}{\sqrt{D_f(Y_0)}} \le \frac{\epsilon R_{f}(Y_0)}{\sqrt{D_f(Y_0)}} \right) = \beta 
\end{equation} 
where $D_f(Y_0)$ is the estimation variance of $\hat{R} _f(Y_0)$ given by \eqref{eq6}. Invoking the Central Limit Theorem,  when $N$ is large enough, we have
\begin{equation}\nonumber
\frac{R_{f}(Y_0)-\hat{R}_{f}(Y_0)}{\sqrt{D_f(Y_0)}} \sim \mathbf{Norm}(0,1)
\end{equation}
In this context, $\epsilon$ can be determined by solving \eqref{eq18}, yielding
\begin{equation}\label{eq19}
\epsilon=\frac{\Phi^{-1}(1/2-\beta/2)\sqrt{D_f(Y_0)}}{R_{f}(Y_0)}
\end{equation}
where $\Phi$ is the cumulative density function of the standard normal distribution, and $\Phi^{-1}$ is the inverse function of $\Phi$.

Eq. \eqref{eq19} says that  a small enough  $D_f(Y_0)$ is required to ensure the credibility of blackout risk estimation. However, as the distribution of $\hat{R}_f(Y_0)$ is unknown, $D_f(Y_0)$ cannot be calculated accurately in practice.  To this end, we give a proposition for estimating $D_f(Y_0)$ based on a specific $Z_g$.

\begin{proposition}
Given $Z_g$, $g(Z)$ and $f(Z)$, an unbiased estimation of $D_f(Y_0)$ is given by
\begin{equation} \label{eq14}
\hat{D}_f(Y_0)=\frac{1}{(N-1)N}\sum\limits_{i=1}^N{\left[w(z^i)h(z^i)\delta_{\{h(z^i) \ge Y_0 \}}-\hat{R}_f(Y_0)\right ]^2}
\end{equation}

\end{proposition}

\begin{proof}

We first define random variable $L_i(Y_0)$ as $L_i(Y_0):= \frac{f(z^i)}{g(z^i)}h(z^i)\delta_{\{h(z^i) \ge Y_0 \}}$, $i=1,\cdots, N$. Since $z^i$ is i.i.d., $L_i(Y_0)$ is i.i.d. as well. Therefore, the variance of random variables $L_i(Y_0)$ are the same. Specifically, we denote the variance of $L_i$ by $d_f(Y_0)$. Then we have  
\begin{equation} \label{eq13}
D_f(Y_0)=\mathbb{D}\left(\frac{1}{N}\sum\limits_{i=1}^N{L_i(Y_0)}\right)
=\frac{1}{N}\mathbb{D}({L_i(Y_0)})
=\frac{1}{N}d_f(Y_0)
\end{equation}   
where $\mathbb{D}$ is an operator of variance calculation. The second equation holds by noting $L_i(Y_0)$,$i=1,\cdots, N$ are i.i.d.. 

According to the definition of sample variance\cite{r19}\fliu{which equation? }, the unbiased estimation of $d_f(Y_0)$ based on $Z_g$ is given by 
\begin{equation} \label{eq16}
\hat{d}_f(Y_0)=\frac{1}{N-1}\sum\limits_{i=1}^N{\left
	[w(z^i)h(z^i)\delta_{\{h(z^i) \ge Y_0 \}}-\hat{R}_f(Y_0)\right]^2}
\end{equation}
Therefore, substituting \eqref{eq16} into \eqref{eq13}, the unbiased estimation of $D_f(Y_0)$ can be given by \eqref{eq14}.
\end{proof}

Combining Eqs. \eqref{eq6}, \eqref{eq19} and \eqref{eq14}, $\epsilon$ can be estimated by
\begin{equation}\label{eq20}
\hat{\epsilon} = \frac{\Phi^{-1}(1/2-\beta/2)\sqrt{\hat{D}_f(Y_0)}}{\hat{R}_{f}(Y_0)}
\end{equation}
And $[(1-\hat{\epsilon})\hat{R}_f(Y_0),(1+\hat{\epsilon})\hat{R}_f(Y_0)]$ can be used as a data-based approximate error bound of $\hat{R}_f(Y_0)$. 

On the other hand, according to Eqs. \eqref{eq19} and \eqref{eq13}, we have

\begin{equation}\label{eq15}
\epsilon \propto \frac{\sqrt{d_f(Y_0)}}{\sqrt{N}R_f(Y_0)}
\end{equation}

Since $d_f(Y_0)$ is deterministic for specific $Z_g$ and $f(Z)$, when the error bound of blackout risk estimation in \eqref{eq9} cannot satisfy the predefined requirement, the sample size $N$ should be enlarged and $Z_g$ should be updated, as we explain next.

\subsection{Determining the Sample Size $N$}
In this subsection, we utilize the previous credibility analyses to determine an appropriate $N$ based on the inference from simulation data in $Z_g$. Basically, if the estimation error bound is determined (according to the requirement of practical utilization), say, $\bar{\epsilon}=10\%$, then \eqref{eq20} gives
\begin{equation}\label{eq21}    
\bar{D}_f(Y_0) =\left (\frac{\bar{\epsilon} \hat{R}_{f}(Y_0)}{\Phi^{-1} (1/2-\beta/2) }\right)^2
\end{equation} 
In Eq. \eqref{eq21}, $\bar{D}_f(Y_0)$ stands for the maximal estimation variance to make the relative error within a given error bound $\bar{\epsilon}$ with a confidence level of $\beta$. 

To ensure that $D_f(Y_0)$ is less  $\bar{D}_f(Y_0)$, the required sample size $\bar{N}$ can be obtained by using \eqref{eq13} and \eqref{eq16}, leading to  
\begin{equation}\label{eq22}
\bar{N}=\frac{\hat{d}_f(Y_0)}{\bar{D}_f(Y_0)}=\frac{\hat{d}_f(Y_0)}{ \hat{R}_{f}^2(Y_0)}\left(\frac{\Phi^{-1} (1/2-\beta/2)}{\bar{\epsilon}}\right)^2
\end{equation}
Eq. \eqref{eq22} indicates that, when $N \le \bar{N} $, the sample size is not large enough to guarantee the credibility of blackout risk estimation and more samples are needed to reduce the estimation variance. It provides the optimization problem \eqref{eq12} with a systematic method to determine an appropriate sample size. The corresponding algorithm  can be easily implemented in practice, which will be presented in  Section \ref{subsec: procedure}. 

It is worthy of nothing that both $\hat{d}_f(Y_0)$ and $\hat{R}_f(Y_0)$ in \eqref{eq22} are calculated only based on the sample set $Z_g$. Hence $\bar{N}$ is essentially inferred from simulation data in $Z_g$. And there is no need to regenerate any additional samples as long as the sample size condition $N>\bar{N}$ holds. 

\section{Solution Algorithms } \label{sec:alg}
\subsection{Heuristic Algorithms}
As mentioned previously, the optimization model \eqref{eq12} is a high-dimensional nonlinear 0-1 programming, and is essentially an NP-hard problem. 
On the other hand, many nearly-optimal solutions are still of great performance to mitigate blackout risk in practice (which will be demonstrated in case studies). In this regard, here we devise two  heuristic algorithms to find nearly-optimal solutions to \eqref{eq12}  on a fixed $Z_g$.    

Since one of the big obstacles is due to the large number of decision variables, an naive thought is to reduce the number of component considered for maintenance, i.e., $|k^*|$. To this end, we employ a sensitivity-based approach to design an algorithm, as presented below.

\noindent\rule[0.25\baselineskip]{0.49\textwidth}{1pt}
{\bf Algorithm I:}

\noindent\rule{0.49\textwidth}{0.5pt}
\begin{itemize}
\small
	\item { \emph{\textbf{Step 1}: Sensitive analysis}}. 
     Construct $|K^*|$ scenarios, in each of which the $\varphi_k$ of a single component $k \in K$ changes to $\bar{\varphi}_k$. Then estimate the blackout risk with \eqref{eq6}.
	\item { \emph{\textbf{Step 2}: Scenario reduction}.}
	Choose the $M_k$ components having larger blackout risk reductions to make up the  $K_m$. Here, $M_k$ is predefined according to $M_{max}$ and experience. Particularly, $M_{max}<M_k=|K_m|<|K^*|$.	
	\item { \emph{\textbf{Step 3}: Solving optimization problem}.} 
	Substitute $K^*$ in \eqref{eq12} by $K_m$ and solve \eqref{eq12}  by enumeration.
\end{itemize}
	
\noindent\rule{0.49\textwidth}{1pt}

Since \eqref{eq6} only needs a few algebraic calculations, the blackout risk estimation of each scenario is very efficient. Suppose the average calculating time of a single scenario is $t_s$, the total calculating time is $(|K^*|+C(M_k,M_{max}))t_s$. Noting that the maximal number of components that are considered for maintenance, i.e., $M_{max}$, is often small in practice, $M_k$ can be chosen as a moderate number. Then, the number of scenarios and the computation time by enumeration is acceptable.  

Nevertheless, when $M_k$ and $M_{max}$ become lager, the scenario number $(|K^*|+C(M_k,M_{max}))$ and the computation time will grow dramatically. In this case, we propose another successive algorithm as shown below.

\noindent\rule[0.25\baselineskip]{0.49\textwidth}{1.05pt}
{\bf Algorithm II:}

\noindent\rule[0.05\baselineskip]{0.49\textwidth}{0.5pt}
\begin{itemize}
\small
	\item { \emph{\textbf{Step 1}: Initialization}}.	
    Let $K_m=\emptyset$. Its complementary set denoted by $\bar{K}_m$  equals $K^*$. 
     
	\item { \emph{\textbf{Step 2}: Blackout risk estimation}}.	
	Construct $|\bar{K}_m|$ scenarios. In each scenario, for all components in $K_m$ and a single component in $\bar{K}_m$, change their $\varphi_k$ to $\bar{\varphi_k}$ and estimate the blackout risk with \eqref{eq6}.  
	
	\item { \emph{\textbf{Step 3}: Iteration}.} 	
	Choose the scenario with the lowest blackout risk. Move corresponding component from $\bar{K}_m$ to $K_m$. If $|K_m|\,=\,M_{max}$, procedure ends; otherwise, go back to Step 2.
\end{itemize}
	
\noindent\rule[0.25\baselineskip]{0.49\textwidth}{1pt}

In  Algorithm II, the components need to be maintained are determined successively. In each round, the component that can reduce the blackout risk most effectively is selected. The total number of scenario is $(2|K^*|-M_{max}+1)M_{max}/2$. Since the scenario number increases linearly with $K^*$, the optimization process remains efficient even in a large  system.

\subsection{Procedure} \label{subsec: procedure}
The procedure of the methodology (see Fig. \ref{fig.2}) is summarized as follows. \fliu{In case we have no enough space, this figure can be removed. }

\noindent\rule[0.25\baselineskip]{0.5\textwidth}{1pt}

{\bf Procedure I:}

\noindent\rule[0.05\baselineskip]{0.5\textwidth}{0.5pt}
\begin{itemize}
\small
	\item {\bf Step 1: Initialization}. Initiate necessary data of the system. Particularly, determine $\varphi_k$ and $\bar{\varphi}_k$ for each component in the system. Moreover, the initial sample size is $N=N_0$. 
	\item {\bf Step 2: Sample generation.}  Generate $N$ samples based on the specific blackout model and $\varphi_k,k\in K$.  The sample set is $Z_g$. Based on $Z_g$, $\varphi_k$ and $\bar{\varphi}_k$, calculate matrices $C$,$P$ and $Q$. 
	\item {\bf Step 3: Optimizing maintenance strategy.} Solving the optimization problem  \eqref{eq12} by Algorithm I or II. 
	\item {\bf Step 4: Credibility evaluation.} For the optimal maintenance strategy, calculate the  necessary sample size $\bar{N}$ with \eqref{eq22}. If $\bar{N} \ge N$, generate another $\bar{N}-N$ samples and add them into $Z_g$. Then go back to Step 3; Otherwise, directly choose the optimal one according to the results in Step 3.  
\end{itemize}
\noindent\rule[0.25\baselineskip]{0.5\textwidth}{1pt}	

\begin{figure}[!t]
	\centering
	\includegraphics[width=0.85\columnwidth]{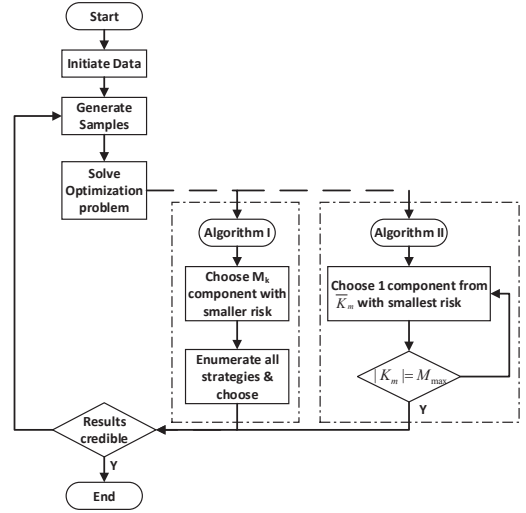}
	\caption{Flowchart of the methodology}
	\label{fig.2}
\end{figure}
Procedure I incorporates an adaptive selection of sample size based on credibility evaluation, as shown in Step 4. In this way, the sample size can be minimized, while the estimation error is guaranteed within a predefined limit. It is worthy of noting that, benefiting from the inference from blackout simulation data, all samples are generated with respect to $g(Z)$ only, other than different $f(Z)$. It leads to a significant reduction of sampling time and simplification of implementation. 

\section{Case Studies}
\subsection{Settings}
In this section, the numerical experiments are carried out on the IEEE 57-bus system and 300-bus system with a simplified OPA model omitting the slow dynamic \cite{r13}. In simulations, traditional MC is employed considering random failures of both transmission lines and power transformers. The failure probability functions of transmission lines and power transformers are the ones in \cite{r16} and \cite{r18}, respectively. Typical parameters are used as well. For simplicity, only the maintenance of power transformers is considered in this work.  

It is worth of noting that the simulation model with specific component failure probability functions mentioned above is just employed to  demonstrate the proposed method. More realistic models and settings can be adopted for coping with more realistic situations. 

\subsection{IEEE 57-bus system}
In this case, we use a small system including 53 transmission lines and 17 power transformers to demonstrate the difficulties in mitigating cascading blackout risk via component maintenance and some salient features  of the proposed method.   

\begin{figure}[!t]
	\centering
	\includegraphics[width=0.85\columnwidth]{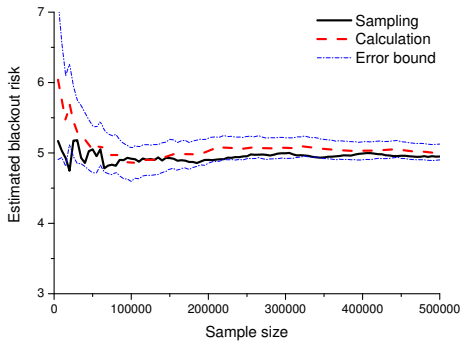}
	\caption{Blackout risk Estimation  with sampling and calculation}
	\label{fig.1}
\end{figure}

\begin{figure}[!t]
	\centering
	\includegraphics[width=0.85\columnwidth]{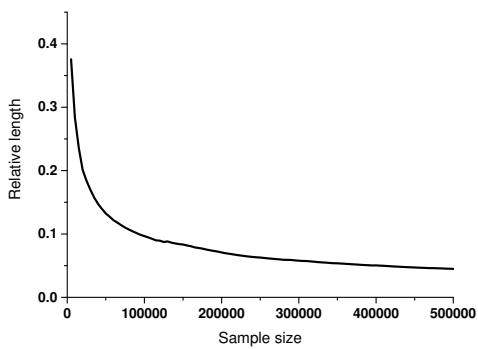}
	\caption{Relative length of the error bounds}
	\label{fig.3}
\end{figure}

\subsubsection{Blackout risk estimation}
We first show the unbiasedness of \eqref{eq6} and its inherent estimation error caused by the randomness of samples. To this end, we construct a series of $Z_g$ with increasing sample size and randomly change the $\varphi_k$ of four components. Then we calculate the blackout risk with $Y_0=200$ by using \eqref{eq6}. Meanwhile, we estimate $\epsilon$ with respect to $\beta=90\%$ by \eqref{eq20} and calculate the corresponding error bounds.  As a comparison, we directly regenerate samples with respect to new $\bar{\varphi}_k$ and estimate the blackout risk by using \eqref{eq5}. The results are presented in Fig. \ref{fig.1}. 

Fig. \ref{fig.1} shows that the estimation errors between calculation and directly sampling become smaller along with the increase of the sample size. When the sample size is large enough, they almost the same. That indicates the effectiveness of the blackout risk estimation given by \eqref{eq6}. To show this more rigorously, we calculate the relative length of the error bounds (see Fig. \ref{fig.3}), which becomes smaller along with the increase of sample size as well.  

However, it is worthy of noting that Fig. \ref{fig.3} also indicates the estimation error always exist. More importantly, when the sample size is small, the estimation error can be very large. Intuitively, the optimized maintenance strategies based on a small $Z_g$ may have a large deviation in performance of mitigating blackout risk. Therefore,  considering the calculation efficiency and the possible estimation error, the sample size of $Z_g$, i.e., $N$,  needs to be selected appropriately.    

\subsubsection{Influence of maintenance on blackout risk }
We construct a $Z_g$ including $100,000$ samples, based on which we consider the simultaneous maintenance of components . Since there are only $17$ transformers in this small system, we directly enumerate all the possible strategies and compare the performance of risk reductions. 
We first consider the maintenance of a single component. The blackout risks with respect to $Y_0=0$ and $Y_0=100$ after maintenance are given in Tab \ref{t2}. Particularly, the original blackout risks are $\hat{R}_g(0)=6.74$ and $\hat{R}_g(100)=5.14$.

\begin{table}[htp]\footnotesize
	\caption{Blackout risk after component maintenance }
	\label{t2}
	\centering
	\begin{tabular}{c|cc|c|cc}
		\hline \hline
        Compon- &Blackout &Reduction &Compon-&Blackout &Reduction\\
        ent index&risk&ratio$(\%)$&ent index&risk&ratio$(\%)$ \\
		\hline
		7  &6.55 &2.8 &7  &4.86 &5.4 \\
		6  &6.64 &1.5 &6  &5.00 &2.7\\
		2  &6.66 &1.1 &3  &5.02 &2.2\\
		3  &6.67 &1.0 &2  &5.02 &2.2\\
		5  &6.67 &1.0 &15 &5.03 &2.1\\
		...&...&...&...&...&...\\
		Mean &6.68&0.9&Mean &5.02&2.2\\
		\hline \hline
	\end{tabular}
\end{table}  

The results in Tab \ref{t2} intuitively indicate that component maintenance is an effective way to mitigate the cascading blackout risk. Moreover, there are a few critical components in the system which have much greater influence on the cascading blackout risk than the others. For example, the maintenance of NO.7 transformer can result in $2.8\%$ reduction of blackout risk with respect to $Y_0=0$, while the average is only $0.9\%$. The result confirms the effecacy of optimizing maintenance strategy in mitigating cascading blackout risks.

Then we consider to maintain four transformers simultaneously. We estimate the blackout risks with all possible maintenance strategies. The strategies with the largest reduction ratios of blackout risks with respect to two $Y_0$ are presented in Tab \ref{t3} and Tab \ref{t4}, respectively. The results show the complicated relationship between maintenance strategies (or more precisely, component failure probabilities) and blackout risk, which is one of the main difficulties. Note that the risk reduction ratios of four components are much smaller than the sum of the four individual ratios. Moreover, the optimal components to be maintained (see Tab \ref{t3} and Tab \ref{t4}) are not the simple combinations of components with smaller blackout risks (see Tab \ref{t2}). It indicates the influence of component failure probabilities on  blackout risk is essentially nonlinear, which implies that one cannot quantify the influence of multiple component maintenance directly based on that of individual component maintenance.

\begin{table}[htp]\footnotesize
	\caption{Optimal component maintenance strategies ($Y_0=0$)}
	\label{t3}
	\centering
	\begin{tabular}{c|ccc}
		\hline \hline
        Compon- &Blackout &Reduction &Sum of individual \\
        ent index&risk&ratio$(\%)$& reduction ratio$(\%)$\\
		\hline
		(7,6,2,5)   &6.505 &3.5 &6.4\\
		(7,6,5,15)  &6.506 &3.5 &6.2\\
		(7,6,2,15)  &6.507 &3.4 &6.4\\
		(7,6,3,5) &6.509 &3.4 &6.3\\
		(7,6,5,14)  &6.511 &3.4 &6.2\\
		...&...&...&...\\
		Mean &6.643&1.4&3.6\\
		\hline \hline
	\end{tabular}
\end{table}  

\begin{table}[htp]\footnotesize
	\caption{Optimal component maintenance strategies ($Y_0=200$)}
	\label{t4}
	\centering
	\begin{tabular}{c|ccc}
		\hline \hline
        Compon- &Blackout &Reduction &Sum of individual \\
        ent index&risk&ratio$(\%)$& reduction ratio$(\%)$\\
		\hline
		(7,6,5,15)   &4.819 &6.2 &12.3\\
		(7,6,2,5)    &4.822 &6.2 &12.3\\
		(7,6,3,15)   &4.823 &6.1 &12.4\\
		(7,6,3,5)    &4.823 &6.1 &12.3\\
		(7,6,2,15)   &4.823 &6.1 &12.4\\
		...&...&...&...\\
		Mean &4.991&2.9&3.6\\
		\hline \hline
	\end{tabular}
\end{table}  

Moreover,  Tab \ref{t3} and Tab \ref{t4} show that there are many nearly-optimal maintenance strategies whose risk reduction ratios are very close to the optimal one, despite the resulting components to be maintained are not exactly the same. Considering the possible estimation error of the blackout risk estimation, these nearly-optimal maintenance strategies can be deployed as well. In other words, there is no need to rigorously solve the optimization problem \eqref{eq12}, particularly when it is very time consuming. Therefore, we employ two heuristic algorithms to efficiently find (nearly-)optimal solutions in this work.   

\subsubsection{Heuristic algorithms}
For the sake of comparing the two heuristic algorithms, we choose different parameters and compare the optimal maintenance strategies obtained by them. Due to the limit of space, here we only present two typical cases. Specifically, $|Z_g|=35000$, $M_{max}=4$. Based on them, we employ the Algorithm I with $M_k=8,10$ and Algorithm II to find the optimal maintenance strategies with respect to $R_f(100)$ and $R_f(200)$, respectively. The results are shown in Tab \ref{t5} and Tab \ref{t6}. 

\begin{table}[htp]\footnotesize
	\caption{Optimal strategies with different algorithms ($Y_0=200$)}
	\label{t5}
	\centering
	\begin{tabular}{c|ccc}
		\hline \hline
		Method  & Strategy  & Risk reduction  & Number of \\
		&&($\%$)&Scenarios \\
		\hline
		Alg. I ($M_k=8$) &(7,6,2,3)& 2.2& 70 \\
		Alg. I ($M_k=10$) &(7,6,2,3)& 2.2& 210 \\
		Alg. II &(7,6,2,3)& 2.2& 62 \\
		Enumeration  &(7,6,2,3)& 2.2& 2380 \\
		\hline \hline
	\end{tabular}
\end{table}

\begin{table}[htp]\footnotesize
	\caption{Optimal strategies with different algorithms ($Y_0=100$)}
	\label{t6}
	\centering
	\begin{tabular}{c|ccc}
		\hline \hline
		Method  & Strategy  & Risk reduction  & Number of \\
		&&($\%$)&Scenarios \\
		\hline
		Alg. I ($M_k=8$)&(6,2,5,14)&1.0&70\\
		Alg. I ($M_k=10$) &(6,2,3,7)&1.4&210\\
		Alg. II&(6,2,5,14)&1.0&62\\
		Enumeration&(6,2,3,7)&1.4&2380\\
		\hline \hline
	\end{tabular}
\end{table}

As shown in Tab \ref{t5}, both algorithms give the same maintenance strategy that achieves the global optimum, while in Tab \ref{t6} Algorithm I with $M_k=8$ and Algorithm II give sub-optimal solutions
.  The  reason relies on the inherent nonlinearity between component failure probabilities and  blackout risk, as well as inevitable uncertainty in samples. However, it is worthy of noting that, according to our experience, Algorithm I with a moderate $M_k$ and Algorithm II can find an effective maintenance strategies in most cases, which can fulfill the requirement of practical applications.  

\subsection{IEEE 300-bus system}
In this case, the complete procedure of the proposed method is tested using the data of IEEE 300-bus system with 
304  lines and 107  transformers. Specifically, we set $\epsilon=10\%$, $\beta=95\%$,  $M_{max}=8$, $M_k=20$, $N_0=5000$, $Y_0=0$. 

\subsubsection{Solving process}
We first present the solving process with Algorithm I (see Tab \ref{t7}).
\begin{table}[htp]\footnotesize
	\caption{Solving process with Algorithm I }
	\label{t7}
	\centering
	\begin{tabular}{c|cccc}
		\hline \hline
		Step  & Sample size & Strategy  & Risk & $\hat{\epsilon}$($\%$)  \\
		\hline
		1 &5000 & (6,17,46,53,68,75,88,106)    & 2.272 &41.9\\
		2 &85000 & (10,25,29,66,68,97,100,106) & 2.007 &10.3\\
		3 &95000 & (10,25,29,66,68,77,100,106) & 2.011 &9.7\\
		\hline \hline
	\end{tabular}
\end{table} 
As shown in Tab \ref{t7}, we start with a $Z_g$ including 5,000 samples. Based on it, the optimal maintenance strategy is obtained. However, the relative error requirement, $\epsilon$ cannot be satisfied. Therefore, $8,0000$ samples are added into $Z_g$ according to \eqref{eq22} in step 2. Then  the optimal maintenance strategy and $\hat{\epsilon}$ are recalculated correspondingly. Repeat this process again in step 3, the final optimal maintenance strategy is determined. The reduction ratio of blackout risk with respect to $Y_0=0$ is $21.5\%$. 

The solving process  with Algorithm II is summarized in  Tab \ref{t8}. In this case, the optimal strategies in each step are actually the same as the ones of Algorithm I. Note that the components of each strategy in Tab \ref{t8} are ordered according to the decision process. Despite the results given by the two algorithms are similar in step 2 and 3, the decision processes are different. This observation demonstrates the complicated relationship between component failure probabilities and blackout risk again.  
\begin{table}[htp]\footnotesize 
	\caption{Solving process with Algorithm II }
	\label{t8}
	\centering
	\begin{tabular}{c|cccc}
		\hline \hline
		
		Step  & Sample size & Strategy  & Risk & $\hat{\epsilon}$($\%$)  \\
		\hline
		1 &5000 &  (46,106,68,53,888,75,6,17)    & 2.272 &41.9\\
		2 &85000 & (106,25,68,66,100,29,10,97) & 2.007 &10.3\\
		3 &95000 & (106,25,68,29,66,100,77,10) & 2.011 &9.7\\
		
		\hline \hline
	\end{tabular}
\end{table} 

\subsubsection{Computation time}
We carry out all tests on a computer with an Intel Xeon E5-2670 of 2.6GHz and 64GB memory. The computation times of the two algorithms are presented in Tab \ref{t9} and Tab \ref{t10}, respectively. Note that the computation time for sampling and calculating $C,P,Q$ in each step depends on the number of additional samples, while the optimizing time depends on the number of total samples. It is observed that the computation time for sampling is much larger than others. In our cases, 107 minutes are required to complete sampling. Therefore,  it is extremely time consuming, if not impossible, to directly estimate all blackout risks with all maintenance strategies by traditional methods. Our methodology enables a very efficient estimation of blackout risk under varying maintenance strategies without regenerating any samples. In addition, the optimizing time of Algorithm II is smaller than the one of Algorithm I as fewer maintenance scenarios are involved. Whereas it may give sub-optimal solutions in some case, e.g., Tab \ref{t5}, it is preferable for large-scale systems with many candidate components considered for maintenance.

\begin{table}[htp]\footnotesize
	\caption{Computation time with Algorithm I (min.)}
	\label{t9}
	\centering
	\begin{tabular}{c|cccc}
		\hline \hline
		Step  & Sampling & Calculating $C,P,Q$  & Optimizing & Total  \\
		\hline
		1 &5.35  & 0.37  &0.08  &5.80\\
		2 &90.95 & 5.92  &0.27  &97.14\\
		3 &10.70 & 0.74  &0.31  &11.75\\
      Sum &107.00 &0.42  &0.66  &114.69\\		
		\hline \hline
	\end{tabular}
\end{table}

\begin{table}[htp]\footnotesize
	\caption{Computation time with Algorithm II (min.) }
	\label{t10}
	\centering
	\begin{tabular}{c|cccc}
		\hline \hline
		Step  & Sampling & Calculating $C,P,Q$  & Optimizing & Total  \\
		\hline
		1 &5.35  & 0.37  &$\ll 0.01$  &5.72\\
		2 &90.95 & 5.92  &$\ll 0.01$  &96.87\\
		3 &10.70 & 0.74  &$\ll 0.01$  &11.44\\
		Sum &107.00 &7.03 &$\ll 0.01$ &114.03\\
		\hline \hline
	\end{tabular}
\end{table}

\section{Conclusion with Remarks}
In this paper, we have devised  an efficient methodology  to optimize component maintenance strategies for effectively mitigating cascading blackout risk, where an analytic relationship between blackout risk estimation and maintenance strategies is revealed throught inference form blackout simulation data.  Theoretical analyses and numerical experiments manifest that:
\begin{enumerate}

\item Some components in the power systems have great influence on the propagation of cascading outages. Reducing their failure probabilities by component maintenance can significantly mitigate the cascading blackout risk. 
\item The variance-based analyses further evaluate the credibility of the risk estimation considering different maintenance strategies. Based on that, the proposed heuristic algorithms can efficiently optimize the component maintenance strategies. 
\end{enumerate}

From the case studies, it is found that the most time-consuming step is the generation of sample set. In this paper, we simply consider a conventional MC method, which is not efficient in large-scale systems. In future work, we hope to introduce some recently developed high-efficiency sampling method, such as Sequential Importance Sampling \cite{r16}, SPLITTING \cite{r20} \fliu{here should we add a reference of SPLITTING?}, to further improve the scalability and efficiency of the methodology. Another ongoing work is to include a more practical formulation of maintenance strategy optimization and corresponding solving algorithms.

\fliu{Please check the format of all references to make them in the correct form required by IEEE Journals.}

\end{document}